%% file: main.tex
\documentclass[letterpaper]{article}
\usepackage[preprint]{aaai2027}
\usepackage[hyphens]{url}
\usepackage{graphicx}
\usepackage{natbib}
\usepackage{booktabs}
\usepackage{amsmath}
\usepackage{amsfonts}
\usepackage{amssymb}
\usepackage{mathtools}
\usepackage{amsthm}
\usepackage{algorithm}
\usepackage{algorithmic}

\theoremstyle{plain}
\newtheorem{theorem}{Theorem}[section]
\newtheorem{proposition}[theorem]{Proposition}
\newtheorem{lemma}[theorem]{Lemma}

\theoremstyle{definition}

\newtheorem{assumption}[theorem]{Assumption}
\theoremstyle{remark}
\newtheorem{remark}[theorem]{Remark}

\newcommand{\E}{\mathbb{E}}
\newcommand{\R}{\mathbb{R}}
\newcommand{\diff}{\mathrm{d}}
\newcommand{\abs}[1]{\left|#1\right|}

\newcommand{\Indc}[1]{\mathbf{1}_{\left\{#1\right\}}}

\newcommand{\bfu}{\mathbf{u}}
\newcommand{\bfx}{\mathbf{x}}
\newcommand{\bfy}{\mathbf{y}}
\newcommand{\bfz}{\mathbf{z}}
\newcommand{\bdtau}{\boldsymbol{\tau}}
\newcommand{\bdv}{\boldsymbol{v}}

\title{Diffusion Quasi-Monte Carlo}
\author{Jianlong Chen\equalcontrib, Yifeng Yu\equalcontrib\corresponding}
\affiliations{Department of Mathematical Sciences, Tsinghua University, Beijing 100084, China\\
\texttt{chen-jl22@mails.tsinghua.edu.cn, yyf22@mails.tsinghua.edu.cn}}

\begin{document}

\maketitle

\begin{abstract}
We study high-dimensional numerical integration with respect to complex target measures using diffusion-based transport maps and randomized quasi-Monte Carlo (RQMC). Score-based diffusion models induce a deterministic \emph{probability flow ODE} that transports a simple prior to the target, suggesting a principled way to transform low-discrepancy points on the unit cube into informative samples.
We construct a cube-to-target map by composing a Gaussian base transformation (the component-wise inverse Gaussian CDF, $\Phi^{-1}$) with an Euler-discretized probability flow ODE. To retain unbiasedness under transport approximation, we formulate integration as importance sampling (IS) on the cube.

Our main result provides verifiable conditions under which the resulting IS integrand satisfies the boundary growth condition, implying an $O(N^{-1+\varepsilon})$ RMSE for scrambled nets. We then establish these conditions for diffusion probability-flow transport under mild bounded-derivative assumptions on the learned vector field, explicitly controlling the boundary singularities introduced by $\Phi^{-1}$. Experiments range from a 2D mixture to 784D images and a 40,960D conditional vorticity-assimilation task; in the latter, blocked scrambled Sobol' sampling reduces the randomization standard deviation of nonlinear accuracy metrics at essentially unchanged online denoising cost. Together, these results give a theoretical and empirical foundation for combining diffusion generative modeling with high-precision RQMC integration.
\end{abstract}

\section{Introduction}

High-dimensional numerical integration is ubiquitous in Bayesian inference, uncertainty quantification, computational physics, and scientific machine learning. A prototypical task is to compute
\[
	\mu \;=\; \E_{\bfx\sim \pi}[f(\bfx)]
	\;=\; \int_{\R^d} f(\bfx)\,\pi(\bfx)\,\diff\bfx,
\]
where $\pi$ is a target density on $\R^d$ and $f$ is an integrand of interest. Standard Monte Carlo (MC) achieves an $O(N^{-1/2})$ rate in root-mean-square error (RMSE), which is often insufficient for high-precision applications.

Quasi-Monte Carlo (QMC) and randomized QMC (RQMC) replace i.i.d.\ sampling on $[0,1]^d$ with low-discrepancy point sets and randomizations such as Owen scrambling, and can substantially reduce variance when the induced integrand on the cube has sufficient mixed regularity \cite{nied:1992,dick:pill:2010,dick:kuo:sloa:2013,Owen2006}. Under suitable conditions, scrambled-net RQMC can approach an $O(N^{-1+\varepsilon})$ RMSE for any $\varepsilon>0$ \cite{Owen2006}. However, applying QMC in non-uniform target spaces requires careful measure transformation: one typically constructs a \emph{transport map} $\bdtau:(0,1)^d\to\R^d$ and rewrites the integral as an expectation over the cube \cite{marzouk2016measuretransport}.

Classical constructions (e.g.\ inverse-CDF and Rosenblatt-type transforms) can be effective in low dimensions, but become difficult to design and differentiate in high dimensions. Even when a high-quality transport is available, its interaction with QMC error analysis is subtle: the map may introduce boundary singularities (through inverse-CDF factors) and may change the effective variation of the cube integrand. Modern RQMC theory therefore often couples transformation design with explicit growth/variation control; see, for example, transformed RQMC analyses over general product spaces and variation-preserving transformations \cite{basu:owen:2015,basu:owen:2016}. Related strands include sequential QMC for sampling-based inference \cite{gerb:chop:2015} and probabilistic perspectives on numerical integration \cite{briol_probabilistic_2015}.

Deep generative models provide data-driven transport constructions in high dimensions. In particular, diffusion probabilistic models and score-based generative models \cite{song2020score} (as well as closely related continuous-time flow formulations such as rectified/straight flows \cite{liu2022flow}) admit deterministic dynamics that map a tractable prior to a complex target. The diffusion framework includes a deterministic probability flow ODE whose solution defines a continuous-time transport from a simple prior to the data distribution. This makes diffusion models natural candidates for defining $\bdtau$ for QMC/RQMC.

The remaining question is theoretical: \emph{does the diffusion-induced transport preserve the boundary regularity required by scrambled-net RQMC theory}? The key obstacle is that the standard Gaussian base transformation from $[0,1]^d$ uses $\Phi^{-1}$, whose derivatives blow up near $0$ and $1$. If left untreated, these boundary singularities can invalidate the derivative-based boundary growth conditions that underlie high-order RMSE bounds \cite{Owen2006}. This paper resolves the issue by adopting an importance sampling view and by verifying transport growth conditions that imply Owen-type boundary growth for the IS integrand on the cube.

\paragraph{Contributions.}
We structure our analysis by first developing a general theory, followed by a rigorous verification for diffusion-based transports. Our main contributions are:
\begin{enumerate}
	\item We develop a convergence analysis for IS-RQMC with a general diffeomorphic transport $\bdtau$, giving sufficient growth conditions on $f$, $\pi$, and $\bdtau$ such that the induced cube integrand satisfies Owen's boundary growth condition and yields an $O(N^{-1+\varepsilon})$ RMSE for scrambled nets \cite{Owen2006}.
	\item We construct a diffusion transport map by composing a Gaussian base transformation with an Euler-discretized probability flow ODE induced by a score-based diffusion model \cite{song2020score}.
	\item We prove that this diffusion transport satisfies the required growth conditions under bounded-derivative assumptions on the learned vector field (equivalently, on the score model and diffusion coefficients), carefully controlling the extra logarithmic factors from $\Phi^{-1}$.
	\item We evaluate diffusion RQMC from 2 to 40,960 dimensions, including conditional super-resolution data assimilation, and separate asymptotic rate gains from practically important constant-factor and randomization-stability gains.
\end{enumerate}

\paragraph{Organization.}
The paper is organized as follows. Section~\ref{sec:background} establishes the preliminaries on QMC, RQMC, boundary growth conditions, importance sampling, and diffusion probability flow ODEs. Section~\ref{sec:transport} details the construction of the diffusion transport and formulates the resulting IS-RQMC estimators. Section~\ref{sec:convergence} presents the general convergence theory for arbitrary transport maps, while Section~\ref{sec:diffusion_growth} verifies the specific growth conditions for our diffusion framework. Finally, Section~\ref{sec:numerics} provides the experimental results, and Section~\ref{sec:conclusions} concludes the work.

\section{Background}
\label{sec:background}

\subsection{Quasi-Monte Carlo and randomized quasi-Monte Carlo}

We recall the basic RQMC setting. For an integrand $h:[0,1]^d\to \R$,
\[
	I(h) = \int_{[0,1]^d} h(\bfu)\,\diff\bfu,
\qquad
	\widehat I_N(h) = \frac{1}{N}\sum_{i=1}^N h(\bfu_i),
\]
where $\{\bfu_i\}_{i=1}^N$ is a low-discrepancy design. Randomization (e.g.\ Owen scrambling) produces randomized points $\{\bfu'_i\}$ which are each marginally uniform, allowing unbiasedness and variance estimation, while retaining low discrepancy almost surely \cite{Owen2006}. We refer to \cite{nied:1992,dick:pill:2010,dick:kuo:sloa:2013} for background.

\subsection{Boundary growth condition}
\label{subsec:GrowthCondition}

We use the standard set-index derivative notation (as in Owen's theory). Let $1\!:d=\{1,\dots,d\}$. For $v\subseteq 1\!:d$, define
\[
	\partial^v := \prod_{j\in v}\frac{\partial}{\partial u_j}.
\]

\begin{assumption}[Boundary growth condition \cite{Owen2006}]
\label{assump:growth}
\begin{equation}\label{eq:owen_growth_form}
	\abs{\partial^v h(\bfu)}
	\;\le\;
	C \prod_{j=1}^d \bigl[\min(u_j,1-u_j)\bigr]^{-B-\Indc{j\in v}}.
\end{equation}
\end{assumption}

\begin{theorem}[Scrambled-net RMSE; adapted from \cite{Owen2006}]
\label{thm:owen_rate}
If $h$ satisfies Assumption~\ref{assump:growth}, then the scrambled net estimator for $I(h)$ achieves RMSE $O(N^{-1+\varepsilon})$ for any $\varepsilon>0$.
\end{theorem}

\subsection{Transport map and importance sampling}
\label{subsec:transport_is}

We want to compute $\mu=\E_{\bfx\sim \pi}[f(\bfx)]$ for a target density $\pi$ on $\R^d$. Let $\bdtau:(0,1)^d\to\R^d$ be a $C^1$ diffeomorphism \cite{marzouk2016measuretransport} and denote the induced proposal density $q_{\bdtau}$ by change of variables:
\begin{equation}\label{eq:cov_formula}
	\begin{aligned}
		q_{\bdtau}(\bfx)
		&= \abs{\det J_{\bdtau}(\bdtau^{-1}(\bfx))}^{-1},\\
		J_{\bdtau}(\bfu)
		&= \left(\frac{\partial \tau_i}{\partial u_j}\right)_{i,j}.
	\end{aligned}
\end{equation}
Even if $q_{\bdtau}\neq \pi$, we may use importance sampling:
\[
	\begin{aligned}
		\mu
		&=\E_{\bfy\sim q_{\bdtau}}\!\left[f(\bfy)\frac{\pi(\bfy)}{q_{\bdtau}(\bfy)}\right]\\
		&=\E_{\bfu\sim \mathrm{Unif}((0,1)^d)}\!\left[
			(f\circ\bdtau)(\bfu)\,(\pi\circ\bdtau)(\bfu)\,\abs{\det J_{\bdtau}(\bfu)}
		\right].
	\end{aligned}
\]
\begin{equation}\label{eq:h_def}
	h(\bfu)
	\;:=\;
	(f\circ\bdtau)(\bfu)\cdot(\pi\circ\bdtau)(\bfu)\cdot \abs{\det J_{\bdtau}(\bfu)}.
\end{equation}
Then $\mu=I(h)$. The IS-RQMC estimator is
\[
	\widehat\mu_N
	\;=\;
	\frac{1}{N}\sum_{i=1}^N h(\bfu_i'),
\]
with scrambled points $\{\bfu_i'\}$, and is unbiased because each $\bfu_i'$ is uniform.

\subsection{Diffusion models and the probability flow ODE}

A diffusion model defines a forward stochastic differential equation (SDE)
\[
	\diff \bfx = \mathbf{f}(\bfx,t)\diff t + g(t)\diff \mathbf{w},
\qquad t\in[0,T],
\]
which transforms data into a tractable prior (often standard Gaussian) as $t\to T$. In this paper we take the terminal distribution $p_T$ to be standard Gaussian on $\R^d$.
This choice is consistent with our sampling interface: if $\bfu\sim\mathrm{Unif}((0,1)^d)$ and $\bfz=G(\bfu)$ with $G_j(u)=\Phi^{-1}(u)$, then $\bfz\sim \mathcal{N}(0,I)$.

Importantly, using a Gaussian prior does not affect the IS-RQMC convergence theory in Sections~\ref{sec:transport}--\ref{sec:convergence}: the theory only requires a well-defined composite map $\bdtau:(0,1)^d\to\R^d$ and bounds on its boundary growth. The prior choice enters only through the base map $G$ (and hence through the derivative/value growth conditions that we verify).
The corresponding score $\nabla_{\bfx}\log p_t(\bfx)$ is approximated by a neural score model $\mathbf{s}_\theta(\bfx,t)$.

The probability flow ODE associated with the diffusion reads
\begin{equation}\label{eq:pf_ode}
	\begin{aligned}
		\frac{\diff\bfx}{\diff t}
		&=
		\mathbf{f}(\bfx,t) - \frac{1}{2}g(t)^2\,\nabla_{\bfx}\log p_t(\bfx)\\
		&\approx
		\mathbf{f}(\bfx,t) - \frac{1}{2}g(t)^2\,\mathbf{s}_\theta(\bfx,t).
	\end{aligned}
\end{equation}
Let us denote the (learned) ODE vector field by
\begin{equation}\label{eq:vtheta_def}
	\bdv_\theta(\bfx,t)
	\;:=\;
	\mathbf{f}(\bfx,t) - \frac{1}{2}g(t)^2\,\mathbf{s}_\theta(\bfx,t).
\end{equation}
Solving the ODE from $t=T$ to $t=0$ maps a prior sample to a sample in data space.

\section{Transport maps for QMC and RQMC with diffusion probability flow}
\label{sec:transport}

We now specialize the transport-map construction to diffusion probability-flow dynamics.

\subsection{Composite transport map}

We construct a map from the cube to $\R^d$ as a composition:
\[
  \bdtau(\bfu)
	\;:=\;
	\widetilde{\bdtau}\bigl(G(\bfu)\bigr),
\qquad
  \bfu\in(0,1)^d.
\]
\begin{enumerate}
	\item $G:(0,1)^d\to\R^d$ maps uniform inputs to a Gaussian prior.
	\item $\tilde{\bdtau}:\R^d\to\R^d$ is the terminal-time flow map generated by the probability flow ODE.
\end{enumerate}

\paragraph{Base map (Gaussian prior).}
We take the component-wise inverse Gaussian CDF map
\begin{equation}\label{eq:probit_map}
	G(\bfu)
	\;=\;
	\bigl(\Phi^{-1}(u_1),\dots,\Phi^{-1}(u_d)\bigr)^\top,
\end{equation}
so that $G(\bfu)\sim \mathcal{N}(0,I)$ when $\bfu\sim \mathrm{Unif}((0,1)^d)$.

\paragraph{ODE flow map.}
Let $\Psi_\theta(\bfx,t)$ be the solution map of
\[
	\partial_t \Psi_\theta(\bfx,t) = \bdv_\theta(\Psi_\theta(\bfx,t),t),
\qquad
	\Psi_\theta(\bfx,T)=\bfx.
\]
Define the terminal map $\tilde{\bdtau}(\bfx):=\Psi_\theta(\bfx,0)$ (integrating backward from $T$ to $0$). Then the full cube-to-target transport is $\bdtau=\tilde{\bdtau}\circ G$.

\subsection{Discretized implementation}

In practice, we discretize the ODE with $N$ steps of size $h=T/N$.
Let $t_k=T-kh$ for $k=0,\dots,N$. The explicit Euler step map is
\begin{equation}\label{eq:euler_step}
	\bdtau^k(\bfx) = \bfx - h\,\bdv_\theta(\bfx,t_k),
\qquad
	k=0,\dots,N-1.
\end{equation}
The discretized flow is the composition
\[
	\begin{aligned}
		\widetilde{\bdtau}^{0:N}(\bfx)
		&=\bdtau^{N-1}\circ\cdots\circ \bdtau^{0}(\bfx),\\
		\bdtau^{0:N}(\bfu)
		&=\widetilde{\bdtau}^{0:N}(G(\bfu)).
	\end{aligned}
\]

\subsection{Diffusion-ISRQMC estimator}

Given scrambled points $\{\bfu_i\}_{i=1}^N$, we evaluate the IS integrand
\[
	\begin{aligned}
		\widehat\mu_N
		&= \frac{1}{N}\sum_{i=1}^N
		f(\bdtau^{0:N}(\bfu_i))\,
		\pi(\bdtau^{0:N}(\bfu_i))\\
		&\qquad\quad{}\times
		\abs{\det J_{\bdtau^{0:N}}(\bfu_i)}.
	\end{aligned}
\]
This is the unbiased importance-sampling estimator on the cube.

\begin{algorithm}[!htbp]
\caption{Diffusion-ISRQMC integration}
\begin{algorithmic}[1]
\REQUIRE integrand $f$, target density $\pi$ (up to normalizing constant if needed), score/diffusion vector field $\bdv_\theta$, scrambled point set $\{\bfu_i\}_{i=1}^N\subset(0,1)^d$, number of steps $N_{\text{step}}$
\FOR{$i=1$ to $N$}
	\STATE $\bfz \gets G(\bfu_i)$ \COMMENT{Inverse Gaussian CDF base map}
	\STATE $\bfx \gets \bfz$
	\FOR{$k=0$ to $N_{\text{step}}-1$}
		\STATE $\bfx \gets \bfx - h\,\bdv_\theta(\bfx,t_k)$
	\ENDFOR
	\STATE $\hat{x}_i \gets \bfx$
	\STATE Evaluate $f(\hat{x}_i)$, $\pi(\hat{x}_i)$, and $\abs{\det J_{\bdtau^{0:N}}(\bfu_i)}$
\ENDFOR
	\STATE \textbf{return} $\widehat\mu_N =$
	\STATE $\frac{1}{N}\sum_{i=1}^N f(\hat{x}_i)\pi(\hat{x}_i)\abs{\det J_{\bdtau^{0:N}}(\bfu_i)}$
\end{algorithmic}
\end{algorithm}

\section{Convergence analysis for general transport maps with QMC}
\label{sec:convergence}

This section develops a general sufficient-condition framework ensuring $h$ in \eqref{eq:h_def} satisfies the boundary growth condition.

\subsection{Assumptions}

We follow a standard pattern: control (i) $f$ growth, (ii) $\pi$ decay, (iii) transport derivative/value growth.

\begin{assumption}[Integrand $f$ growth]\label{assum:f}
Assume there exists a univariate CDF $F:\R\to(0,1)$ such that for all multi-indices $\lambda\in\mathbb{N}_0^d$ with $|\lambda|\le d$, and for arbitrarily small $B>0$, there exists $C_1>0$ such that for all $\bfx\in\R^d$,
\begin{equation}\label{eq:f_growth}
	\abs{\partial^\lambda f(\bfx)}
	\;\le\;
	C_1 \prod_{k=1}^d \Bigl[\min(F(x_k),1-F(x_k))\Bigr]^{-B}.
\end{equation}
\end{assumption}

\begin{assumption}[Target $\pi$ decay]\label{assum:pi}
Assume $\pi$ and its derivatives have Gaussian-type tail decay in the $F$-tails. Concretely, there exist $C_2>0$ and $\alpha>0$ such that for all multi-indices $\lambda$ with $|\lambda|\le d$,
\begin{equation}\label{eq:pi_decay}
	\abs{\partial^\lambda \pi(\bfx)}
	\;\le\;
	C_2 \prod_{k=1}^d e^{-\alpha x_k^2}.
\end{equation}
\end{assumption}

\begin{assumption}[Transport growth]\label{assum:tau}
Let the same $F$ as in Assumption~\ref{assum:f} be fixed, and write
$\rho(x):=\min(F(x),1-F(x))$. Assume that for any $v\subseteq 1\!:d$, any $m\in 1\!:d$ and any component $j\in 1\!:d$, and for arbitrarily small $B>0$, there exists $C_3>0$ such that for all $\bfu\in(0,1)^d$,
\begin{equation}\label{eq:tau_deriv_growth}
	\begin{aligned}
		\abs{\partial^v(\partial^{\{m\}}\tau_j)(\bfu)}
		&\le
		C_3\, \min(u_m,1-u_m)^{-1}\\
		&\quad\cdot
		\prod_{k=1}^d \min(u_k,1-u_k)^{-B-\Indc{k\in v}}.
	\end{aligned}
\end{equation}
Moreover, there exist $B_0>0$ and $C_4>0$ such that for any $k\in 1\!:d$,
\begin{equation}\label{eq:tau_value_growth}
	\begin{aligned}
		\rho(\tau(\bfu)_k)^{-1}
		&\le
		C_4\prod_{\ell=1}^d \min(u_\ell,1-u_\ell)^{-B_0}.
	\end{aligned}
\end{equation}
Finally, assume there exist $C_5>1/\alpha$ and $C'\in\R$ such that
\begin{equation}\label{eq:tau_log_lower}
	\begin{aligned}
		\abs{\tau(\bfu)_j}^2
		&\ge
		C_5\,\abs{\ln(\min(u_j,1-u_j))} - C'.
	\end{aligned}
\end{equation}
\end{assumption}

\begin{remark}
Assumption~\ref{assum:tau} has three roles: \eqref{eq:tau_deriv_growth} controls Jacobian / derivative blow-up near the cube boundary; \eqref{eq:tau_value_growth} allows us to translate tail growth of $f$ (measured in $F$-coordinates) back into boundary coordinates; \eqref{eq:tau_log_lower} ensures $\pi\circ\tau$ contributes enough decay to offset Jacobian singularities.
\end{remark}

\subsection{Growth analysis of the three factors}

Recall
\[
	h(\bfu)= (f\circ\bdtau)(\bfu)\cdot(\pi\circ\bdtau)(\bfu)\cdot \abs{\det J_{\bdtau}(\bfu)}.
\]
We now bound the mixed derivatives of each factor. The bounds for $f\circ\bdtau$ and $\abs{\det J_{\bdtau}}$ coincide with Propositions~8 and~10 in closely related prior work \cite{zeng2026flowmatching}, so we cite them directly. The bound for $\pi\circ\bdtau$ is similar but not identical to the corresponding result in that work, so we include a short proof.

\begin{proposition}[Growth of $f\circ\bdtau$]\label{prop:fcirc_tau}
Under Assumptions~\ref{assum:f} and \ref{assum:tau}, for arbitrarily small $B>0$, there exists $C>0$ such that for all $v\subseteq 1\!:d$ and $\bfu\in(0,1)^d$,
\[
	\abs{\partial^v(f\circ\bdtau)(\bfu)}
	\;\le\;
	C \prod_{k=1}^d \min(u_k,1-u_k)^{-B-\Indc{k\in v}}.
\]
\end{proposition}

\noindent\textit{Reference.} This is Proposition~8 in \cite{zeng2026flowmatching}.

\begin{proposition}[Growth of $\pi\circ\bdtau$]\label{prop:picirc_tau}
Under Assumptions~\ref{assum:pi} and \ref{assum:tau}, for arbitrarily small $B>0$, there exists $C>0$ such that for all $v\subseteq 1\!:d$,
\[
	\abs{\partial^v(\pi\circ\bdtau)(\bfu)}
	\;\le\;
	C \prod_{k=1}^d \min(u_k,1-u_k)^{1-B-\Indc{k\in v}}.
\]
\end{proposition}

\begin{proof}
Fix $v\subseteq 1\!:d$.
By the multivariate Fa\`a di Bruno formula (in set-derivative notation), $\partial^v(\pi\circ\bdtau)$ is a finite sum (with the number of terms depending only on $d$) of products of the form
\[
T
=
(\partial^\lambda \pi)(\bdtau(\bfu))\cdot \prod_{r=1}^{|\lambda|} \partial^{v_r}\tau_{j_r}(\bfu),
\]
where $1\le |\lambda|\le |v|$, the sets $v_1,\dots,v_{|\lambda|}$ form a partition of $v$, and $j_r\in 1\!:d$.

For the $\pi$-factor, Assumption~\ref{assum:pi} gives
\[
\abs{(\partial^\lambda \pi)(\bdtau(\bfu))}
\le C\prod_{k=1}^d \exp\bigl(-\alpha\,\tau(\bfu)_k^2\bigr).
\]
Using the lower bound \eqref{eq:tau_log_lower}, we have
$\exp(-\alpha\tau(\bfu)_k^2)\le C\,\min(u_k,1-u_k)^{\alpha C_5}$ for each $k$, hence
\[
\abs{(\partial^\lambda \pi)(\bdtau(\bfu))}
\le C\prod_{k=1}^d \min(u_k,1-u_k)^{\alpha C_5}.
\]

For the transport-derivative product, apply Assumption~\ref{assum:tau} termwise: for each nonempty block $v_r$, pick any $\ell_r\in v_r$ and write
$\partial^{v_r}\tau_{j_r}=\partial^{v_r\setminus\{\ell_r\}}(\partial^{\{\ell_r\}}\tau_{j_r})$.
Then \eqref{eq:tau_deriv_growth} bounds each factor by
$C\,m_{\ell_r}^{-1}\prod_{k=1}^d \min(u_k,1-u_k)^{-B-\Indc{k\in v_r\setminus\{\ell_r\}}}$.
Multiplying over $r$ and using $m_{\ell_r}^{-1}\le \prod_{k\in v_r} \min(u_k,1-u_k)^{-1}$ gives
\[
\prod_{r=1}^{|\lambda|}\abs{\partial^{v_r}\tau_{j_r}(\bfu)}
\le
C\prod_{k=1}^d \min(u_k,1-u_k)^{-B-\Indc{k\in v}}
\]
after reabsorbing constants and shrinking $B>0$ if needed (since $|\lambda|\le |v|\le d$ is fixed).

Combining the two displays yields
\[
\abs{T}
\le
C\prod_{k=1}^d \min(u_k,1-u_k)^{\alpha C_5-B-\Indc{k\in v}}.
\]
Because $C_5>1/\alpha$ in Assumption~\ref{assum:tau}, we have $\alpha C_5>1$, and since $\min(u_k,1-u_k)\le 1$ this implies $\min(u_k,1-u_k)^{\alpha C_5}\le \min(u_k,1-u_k)$.
Therefore $\abs{T}\le C\prod_{k=1}^d \min(u_k,1-u_k)^{1-B-\Indc{k\in v}}$. Summing finitely many terms preserves the same bound (up to constants), proving the claim.
\end{proof}

\begin{proposition}[Growth of $\abs{\det J_{\bdtau}}$]\label{prop:det_growth}
Under Assumption~\ref{assum:tau}, for arbitrarily small $B>0$, there exists $C>0$ such that for all $v\subseteq 1\!:d$,
\[
	\abs{\partial^v \abs{\det J_{\bdtau}(\bfu)}}
	\;\le\;
	C \prod_{k=1}^d \min(u_k,1-u_k)^{-1-B-\Indc{k\in v}}.
\]
\end{proposition}

\noindent\textit{Reference.} This is Proposition~10 in \cite{zeng2026flowmatching}.

\subsection{Main convergence theorem}

\begin{theorem}[IS integrand satisfies boundary growth]\label{thm:main_convergence}
Under Assumptions~\ref{assum:f}, \ref{assum:pi}, \ref{assum:tau}, the integrand $h$ defined in \eqref{eq:h_def} satisfies Assumption~\ref{assump:growth}. Consequently, the scrambled-net IS-RQMC estimator has RMSE $O(N^{-1+\varepsilon})$ for any $\varepsilon>0$.
\end{theorem}

\begin{proof}
Fix $v\subseteq 1\!:d$. By the product rule (Leibniz rule),
$\partial^v h$ is a finite sum of terms of the form
\[
	T
	=
	\partial^{v_1}(f\circ\bdtau)\cdot
	\partial^{v_2}(\pi\circ\bdtau)\cdot
	\partial^{v_3}\abs{\det J_{\bdtau}},
\]
where $(v_1,v_2,v_3)$ is a partition of $v$.

Using Propositions~\ref{prop:fcirc_tau}, \ref{prop:picirc_tau}, \ref{prop:det_growth},
\[
\begin{aligned}
\abs{T}
&\le
C
\prod_{k=1}^d \min(u_k,1-u_k)^{-B-\Indc{k\in v_1}}\\
&\quad\cdot
\prod_{k=1}^d \min(u_k,1-u_k)^{1-B-\Indc{k\in v_2}}\\
&\quad\cdot
\prod_{k=1}^d \min(u_k,1-u_k)^{-1-B-\Indc{k\in v_3}}.
\end{aligned}
\]
The $+1$ and $-1$ cancel between the $\pi\circ\tau$ and determinant factors, leaving
\[
	\abs{T}
	\le
	C
	\prod_{k=1}^d
	\min(u_k,1-u_k)^{-B-\sum_{r=1}^3\Indc{k\in v_r}}.
\]
Since $(v_1,v_2,v_3)$ partitions $v$, we have
$\Indc{k\in v_1}+\Indc{k\in v_2}+\Indc{k\in v_3}=\Indc{k\in v}$.
Thus
\[
	\abs{T}
	\le
	C \prod_{k=1}^d \min(u_k,1-u_k)^{-B-\Indc{k\in v}}.
\]
Summing finitely many terms preserves the same bound (up to constants). This is exactly Assumption~\ref{assump:growth}. The RMSE claim follows from Theorem~\ref{thm:owen_rate}.
\end{proof}

\section{Growth condition verification for diffusion transport map}
\label{sec:diffusion_growth}

We now verify Assumption~\ref{assum:tau} for the diffusion transport map
\[
	\bdtau^{0:N}(\bfu) = \bdtau^{N-1}\circ\cdots\circ \bdtau^0 \circ G(\bfu),
\]
where $G=\Phi^{-1}$ is applied component-wise and $\bdtau^k(\bfx)=\bfx-h\bdv_\theta(\bfx,t_k)$.

\subsection{Setup and assumptions on the vector field}

\begin{assumption}[Bounded diffusion vector field and derivatives]\label{assum:v_bounded}
Assume that $\bdv_\theta(\bfx,t)$ and its spatial derivatives are uniformly bounded up to order $d+1$. Formally, for any multi-index $\alpha\in\mathbb{N}_0^d$ with $0\le |\alpha|\le d+1$, there exists $M_\alpha<\infty$ such that for all components $j\in 1\!:d$,
\[
	\sup_{\bfx\in\R^d,\;t\in[0,T]}
	\abs{\partial_{\bfx}^{\alpha} v_{\theta,j}(\bfx,t)}
	\le M_\alpha.
\]
In particular, taking $|\alpha|=0$ yields global boundedness of $\bdv_\theta$, and taking $|\alpha|=1$ yields global Lipschitz continuity in $\bfx$ (both uniformly in $t$).
\end{assumption}

This is a sufficient regularity condition (and stronger than what is needed in some diffusion analyses). It mirrors the ``bounded velocity field'' condition used in flow-matching-style proofs and can be enforced, for example, by using bounded activations or explicit constraints/regularization.

\subsection{Auxiliary bounds for the Gaussian base map}

We collect key inequalities for $\Phi^{-1}$ that we will reuse.

\begin{lemma}[Quantile growth]\label{lem:quantile_growth}
There exist constants $c_1,c_2>0$ such that for all $u\in(0,1/2]$,
\[
	c_1 \sqrt{\abs{\ln u}} \le \abs{\Phi^{-1}(u)} \le c_2 \sqrt{\abs{\ln u}}.
\]
An analogous bound holds for $u\in[1/2,1)$ by symmetry.
\end{lemma}

\begin{lemma}[Derivative growth of the inverse Gaussian CDF]\label{lem:probit_derivs}
For each integer $k\ge 1$, there exist constants $C_k>0$ such that for all $u\in(0,1)$,
\begin{equation}\label{eq:probit_deriv_bd}
	\begin{split}
		\abs{(\Phi^{-1})^{(k)}(u)}
		\le
		C_k&\, \bigl[\min(u,1-u)\bigr]^{-k} \\ 
		&\cdot
		\abs{\ln(\min(u,1-u))}^{(k-1)/2}.
	\end{split}
\end{equation}
In particular, for any $\varepsilon>0$ there exists $C_{k,\varepsilon}$ such that
\[
	\abs{(\Phi^{-1})^{(k)}(u)}
	\le C_{k,\varepsilon}\, \bigl[\min(u,1-u)\bigr]^{-k-\varepsilon}.
\]
\end{lemma}

\begin{proof}
These bounds are standard consequences of Mills ratio asymptotics and repeated differentiation of the inverse CDF; the logarithmic factor arises because $\phi(\Phi^{-1}(u))\sim u\sqrt{|\ln u|}$ in the tail. The last inequality follows since for any $\varepsilon>0$, $\abs{\ln u}^{a}\le u^{-\varepsilon}$ for sufficiently small $u$, and similarly near $1$.
\end{proof}

\begin{lemma}[Sub-polynomial exponential bound]\label{lem:subpoly}
For any $c>0$ and any $\varepsilon>0$, there exists $C>0$ such that for all $u\in(0,1/2]$,
\[
	\exp\!\bigl(c\sqrt{\abs{\ln u}}\bigr) \le C\,u^{-\varepsilon}.
\]
\end{lemma}

\begin{proof}
Let $s=\sqrt{\abs{\ln u}}$, so $u=e^{-s^2}$. Then $\exp(cs)=\exp(\varepsilon s^2)\exp(cs-\varepsilon s^2)\le \exp(\varepsilon s^2)\cdot \sup_{s\ge 0}\exp(cs-\varepsilon s^2)$. The supremum is finite, and $\exp(\varepsilon s^2)=u^{-\varepsilon}$.
\end{proof}

\subsection{Value growth verification}

We first show that the diffusion map satisfies a version of \eqref{eq:tau_value_growth} and \eqref{eq:tau_log_lower}. The proof follows the same organization as the flow-matching value growth theorem, and uses Lemma~\ref{lem:subpoly} to handle Gaussian tails.

\begin{theorem}[Value growth of diffusion transport]\label{thm:diff_value_growth}
Let $\bdtau^{0:N}(\bfu)=\tilde{\bdtau}^{0:N}(G(\bfu))$ be the discretized diffusion transport with $G=\Phi^{-1}$ applied component-wise and Euler flow maps \eqref{eq:euler_step}. Under Assumption~\ref{assum:v_bounded}, there exist constants $B_0>0$ and $C_4>0$ such that for all $\bfu\in(0,1)^d$ and all $k\in 1\!:d$,
\[
	\begin{aligned}
	&\min\!\bigl(\Phi(\tau^{0:N}(\bfu)_k),\,1-\Phi(\tau^{0:N}(\bfu)_k)\bigr)^{-1}\\
	\le
	&C_4 \prod_{\ell=1}^d \min(u_\ell,1-u_\ell)^{-B_0}.
	\end{aligned}
\]
Moreover, there exist $C_5>0$ and $C'$ such that
\[
	\abs{\tau^{0:N}(\bfu)_k}^2
	\ge
	C_5\,\abs{\ln(\min(u_k,1-u_k))} - C'.
\]
\end{theorem}

\begin{proof}
Since $\bdv_\theta$ is uniformly bounded by $M$, the cumulative displacement of the Euler discretization is bounded by $TM$. Thus, for $\bfx=G(\bfu)$, we have
\begin{equation}\label{eq:bounded_shift}
    \abs{\tau^{0:N}(\bfu)_k - \Phi^{-1}(u_k)} \le TM, \quad \forall k.
\end{equation}
\textbf{Tail growth.} Using \eqref{eq:bounded_shift} and Lemma~\ref{lem:quantile_growth}, we have $\abs{\tau^{0:N}(\bfu)_k} \ge c_1\sqrt{|\ln \min(u_k,1-u_k)|} - TM$, where $\min(u_k,1-u_k) = \min(u_k, 1-u_k)$. Squaring this lower bound and absorbing constants yields the second claim: $\abs{\tau^{0:N}(\bfu)_k}^2 \ge C_5 |\ln \min(u_k,1-u_k)| - C'$.

\textbf{Inverse-tail control.} For the first claim, consider the lower tail $u_k \le 1/2$. Eq.~\eqref{eq:bounded_shift} implies $\Phi(\tau_k) \ge \Phi(\Phi^{-1}(u_k) - TM)$.
Since Gaussian tails decay super-exponentially, the ratio $\Phi(z-C)/\Phi(z)$ behaves sub-polynomially.
By Lemma~\ref{lem:subpoly}, for any $B_0 > 0$, we have $\frac{u_k}{\Phi(\Phi^{-1}(u_k)-TM)} \le C u_k^{-B_0}$.
Hence $\Phi(\tau_k)^{-1} \le C u_k^{-1-B_0}$.
Symmetry handles the upper tail.
Bounding the component-wise minimum by the product over all dimensions completes the proof.
\end{proof}

\subsection{Derivative growth verification}

We now prove the analogue of the induction-based derivative growth theorem.

\begin{theorem}[Derivative growth of diffusion transport]\label{thm:diff_deriv_growth}
Under Assumption~\ref{assum:v_bounded}, the discretized diffusion transport map $\bdtau^{0:N}$ satisfies the derivative growth condition \eqref{eq:tau_deriv_growth} in Assumption~\ref{assum:tau} (for arbitrarily small $B>0$).
\end{theorem}

\begin{proof}
Deferred to Appendix~\ref{app:proof_thm_diff_deriv_growth}.
\end{proof}

\subsection{Putting everything together}

\begin{theorem}[Diffusion-ISRQMC RMSE rate]\label{thm:diffusion_final}
Assume:
(i) $f$ satisfies Assumption~\ref{assum:f};
(ii) $\pi$ satisfies Assumption~\ref{assum:pi};
(iii) the diffusion vector field satisfies Assumption~\ref{assum:v_bounded}.
Then the diffusion transport $\bdtau^{0:N}$ satisfies Assumption~\ref{assum:tau}, and the IS integrand $h$ satisfies the boundary growth condition (Assumption~\ref{assump:growth}). Consequently, the scrambled-net IS-RQMC estimator achieves RMSE $O(N^{-1+\varepsilon})$ for any $\varepsilon>0$.
\end{theorem}

\begin{proof}
The value growth part follows from Theorem~\ref{thm:diff_value_growth}. The derivative growth part follows from Theorem~\ref{thm:diff_deriv_growth}. Thus Assumption~\ref{assum:tau} holds for $\bdtau^{0:N}$. Then Theorem~\ref{thm:main_convergence} implies $h$ satisfies boundary growth and Theorem~\ref{thm:owen_rate} gives the RMSE rate.
\end{proof}

\begin{remark}
Omitting importance sampling gives the transport estimator $\widehat\mu_N^{\mathrm{tr}}:=\frac{1}{N}\sum_{i=1}^N (f\circ\bdtau^{0:N})(\bfu_i)$, whose RMSE still decays as $O(N^{-1+\varepsilon})$ for any $\varepsilon>0$ by Proposition~\ref{prop:fcirc_tau} and Theorem~\ref{thm:owen_rate}.
However, $\widehat\mu_N^{\mathrm{tr}}$ is generally biased for $\E_{\bfx\sim\pi}[f(\bfx)]$ unless the pushforward of the uniform measure under $\bdtau^{0:N}$ equals $\pi$.
\end{remark}

\section{Experiments}
\label{sec:numerics}

\subsection{GMM and MNIST}

We evaluate four estimators on a two-dimensional, two-component Gaussian mixture model (GMM): diffusion transport with either IID Monte Carlo (MC) or scrambled Sobol' RQMC points, with and without importance sampling (IS). The target is
$0.5\mathcal{N}(3\mathbf{e},I)+0.5\mathcal{N}(-3\mathbf{e},I)$,
where $\mathbf{e}=(1,1)^\top$, and the exact value of
$\mathbb{E}[\|X\|^2]$ is 20. The score network is an MLP with four residual blocks and hidden width 256, trained for 1,000 epochs; the probability-flow ODE uses 2,000 Euler steps. We report RMSE over 20 independent randomizations.

Figure~\ref{fig:convergence}a shows that IS removes the error floor caused by approximate transport. Diffusion-ISRQMC attains a log--log slope of approximately $-0.95$ and RMSE $0.002$ at $N=8192$, whereas Diffusion-ISMC has slope approximately $-0.48$ and RMSE approximately $0.08$. Thus, the unbiased RQMC estimator exhibits the rate predicted by our analysis and improves the RMSE by about $40\times$ at the largest sample size.

We also test the transport-only estimators on MNIST \cite{lecun1998gradient} in 784 dimensions using a residual U-Net score model, trained for 10,000 epochs, again with 2,000 Euler steps. Figure~\ref{fig:convergence}b reports the standard deviation of the pixel-mean statistic over 20 randomizations. Diffusion-RQMC follows a near-$N^{-1}$ trend and consistently improves over Diffusion-MC, which follows the standard $N^{-1/2}$ trend. We omit IS here because its weights collapse in this high-dimensional setting; this experiment isolates variance reduction by RQMC rather than bias correction. Additional implementation details, generated samples, and the full derivative-growth proof appear in the appendices.

\begin{figure*}[t]
    \centering
    \begin{minipage}[t]{0.48\textwidth}
        \centering
        \includegraphics[width=\linewidth]{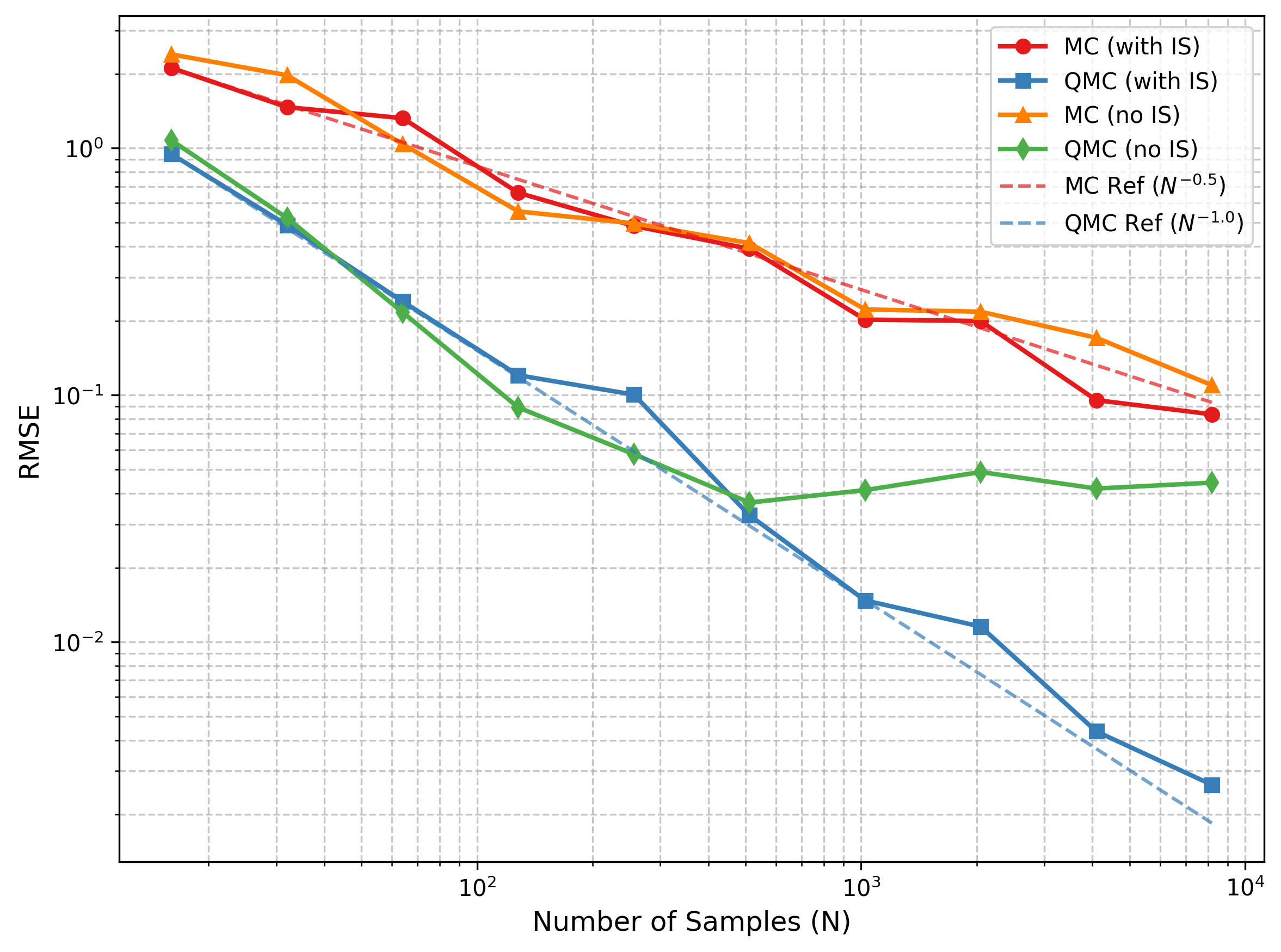}\\
        (a) Two-dimensional GMM: RMSE
    \end{minipage}
    \hfill
    \begin{minipage}[t]{0.48\textwidth}
        \centering
        \includegraphics[width=\linewidth]{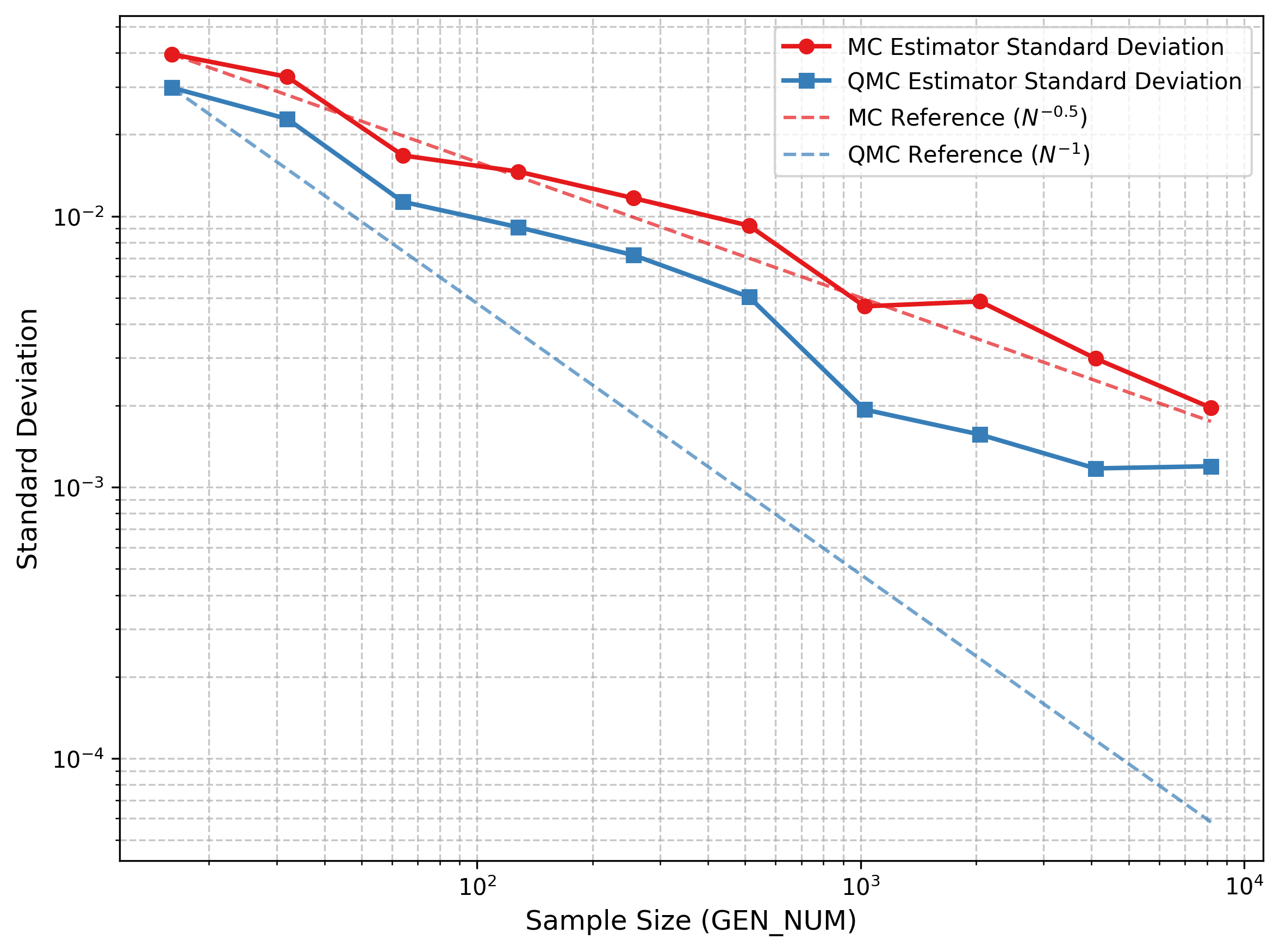}\\
        (b) MNIST: standard deviation
    \end{minipage}
    \caption{Convergence of diffusion-based MC and RQMC estimators. (a) RMSE on the two-dimensional GMM, where importance sampling removes transport bias and Diffusion-ISRQMC approaches the predicted $N^{-1}$ rate. (b) Randomization standard deviation of the MNIST pixel-mean functional, where transport-only RQMC consistently improves over IID MC.}
    \label{fig:convergence}
\end{figure*}

\begin{figure*}[t]
    \centering
    \includegraphics[width=0.94\textwidth]{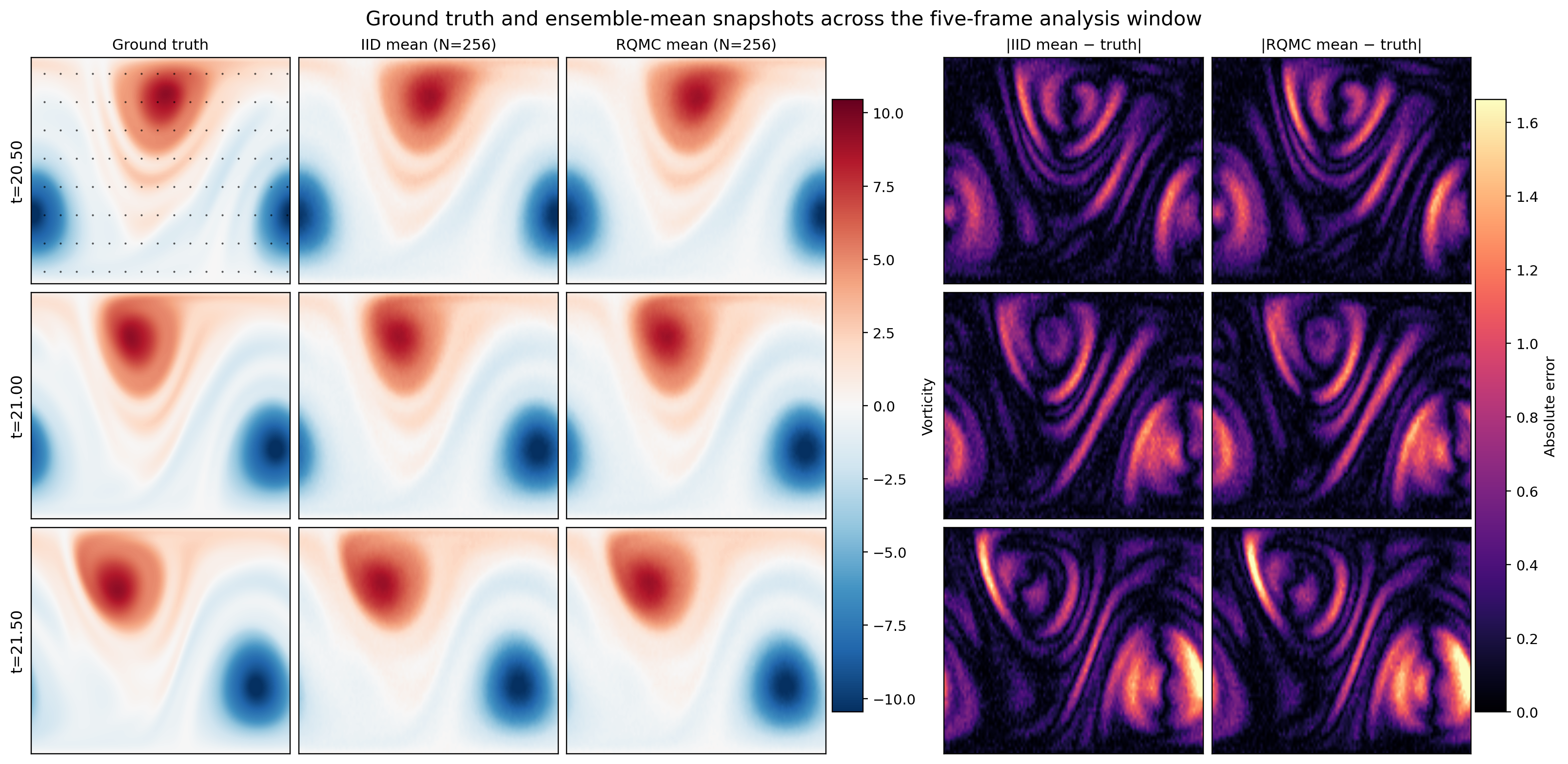}
    \caption{Ground truth, $N=256$ ensemble means, and absolute errors at three forecast leads for a representative held-out vorticity-assimilation context. Dots mark the 128 sparse sensors. Both methods recover the evolving jet structure; the single displayed scramble is qualitative and is not expected to reduce every pointwise truth error.}
    \label{fig:weather_snapshots}
\end{figure*}

\begin{figure*}[t]
    \centering
    \includegraphics[width=0.86\textwidth]{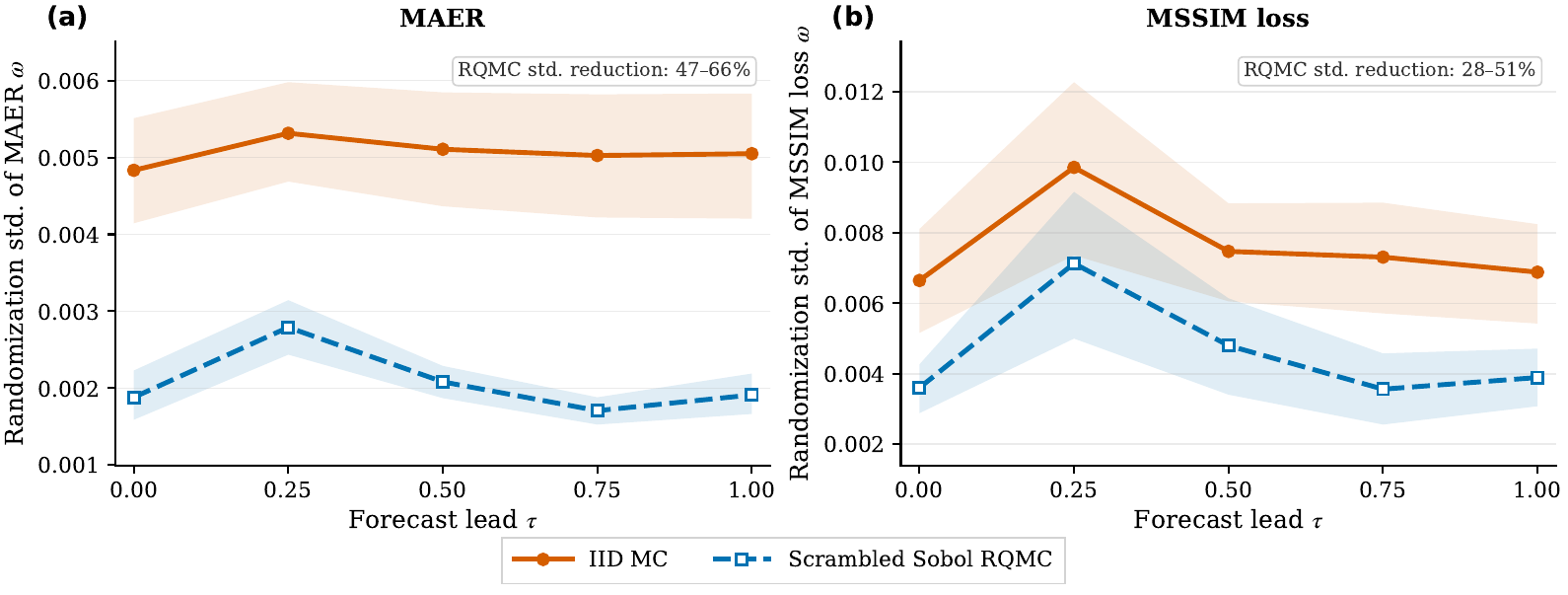}
    \caption{Root-mean randomization standard deviation (RSD) for MAER and MSSIM loss at $N=256$ across the five-frame forecast window. Curves aggregate 20 held-out contexts and 20 independent randomizations; bands are 95\% context-cluster bootstrap intervals. Lower is better.}
    \label{fig:weather_stability}
\end{figure*}

\subsection{Conditional Vorticity Data Assimilation}

We next consider a scientific conditional-generation task adapted from DiffSRDA \cite{ayapilla2026uncertainty}. A conditional diffusion model reconstructs five high-resolution vorticity frames of a synthetic barotropic jet from three low-resolution frames and two sparse-observation frames. Removing the duplicated periodic endpoint gives an output state of size $5\times128\times64=40{,}960$. We train a 128.8M-parameter conditional U-Net and evaluate its best validation checkpoint using deterministic DDIM sampling with five reverse steps and $\eta=0$.

Because 40,960 exceeds the supported Sobol' dimension of 21,201, we split the initial Gaussian latent into two contiguous blocks of dimensions 21,201 and 19,759 and scramble the blocks independently. We therefore refer to this construction as \emph{blocked-Sobol' RQMC}, rather than a single 40,960-dimensional net. Our formal stability evaluation uses 20 held-out contexts, 20 independent randomizations, and $N=256$ samples per ensemble.

At $N=256$, Figure~\ref{fig:weather_stability} shows that blocked-Sobol' RQMC reduces randomization standard deviation by 47--66\% for mean absolute error ratio (MAER) and 28--51\% for MSSIM loss across the five forecast leads; Appendix~\ref{app:supp_weather} reports tabulated reductions and the negative discrete-Laplacian result (VRF 0.89--0.99). With one cached noise bank per replicate reused across contexts and inverse-CDF evaluation on the GPU, online wall time is 2.306\,s for both IID and RQMC; including amortized noise generation gives 2.309\,s and 2.382\,s, respectively, on one NVIDIA H20 (3.18\% overhead). Figure~\ref{fig:weather_snapshots} visualizes context 19. Both methods recover the evolving jet structure; the single displayed scramble is qualitative and is not expected to reduce every pointwise truth error.

\section{Conclusions}
\label{sec:conclusions}

This paper establishes a rigorous link between diffusion probability-flow transport and high-order randomized quasi-Monte Carlo integration. Within an importance sampling formulation, we provided general sufficient conditions on a transport map that guarantee the induced cube integrand satisfies Owen's boundary growth condition, yielding an $O(N^{-1+\varepsilon})$ RMSE for scrambled nets. We then verified these conditions for diffusion transports built from a Gaussian base map (the inverse Gaussian CDF, $\Phi^{-1}$) and an Euler-discretized probability flow ODE, under bounded-derivative assumptions on the learned vector field. Across GMM, MNIST, and conditional vorticity experiments, RQMC reduces integration error or randomization sensitivity. In the 40,960D conditional task, blocked-Sobol' substantially stabilizes MAER and MSSIM at near-identical online cost, while providing no benefit for the derivative-sensitive discrete-Laplacian metric.

Several directions remain open: extending the proof to adaptive ODE solvers and higher-order integrators; relaxing the boundedness assumptions to allow mild time singularities; and tracking constants explicitly in dimension and step count. These questions are important for bridging theory and practical diffusion samplers used in modern scientific computing.
 
\appendix

\input{appendix}

\bibliography{references_arxiv}

\end{document}

%% file: appendix.tex
\section{Additional Experimental Results}
\label{app:additional_results}

To evaluate Diffusion-ISRQMC and its transport-only counterpart across qualitatively different regimes, we consider three settings: a 2D two-component Gaussian mixture model (GMM), MNIST in 784 dimensions \cite{lecun1998gradient}, and conditional super-resolution data assimilation for a synthetic vorticity field in 40,960 dimensions. The GMM experiment isolates importance-sampling correction, MNIST tests an unconditional image model, and the vorticity experiment tests randomization stability and cost in a large conditional scientific model. For GMM and MNIST, we use the following estimators:
\begin{itemize}
    \item Diffusion-MC: Uses independent and identically distributed (IID) uniform sampling on $[0,1]^d$ with diffusion transport, suffering from transport approximation bias.
    \item Diffusion-RQMC: Adopts scrambled Sobol net sampling with diffusion transport, reducing sampling variance but retaining transport bias.
    \item Diffusion-ISMC: Combines IID sampling, diffusion transport, and IS, achieving unbiased estimation with standard MC convergence rate $O(N^{-1/2})$.
    \item Diffusion-ISRQMC: Our proposed method, integrating scrambled Sobol net sampling, diffusion transport, and IS, which is unbiased and theoretically achieves $O(N^{-1+\varepsilon})$ convergence.
\end{itemize}

For vorticity assimilation we compare transport-only IID MC with independently scrambled blocked-Sobol' RQMC. Importance weights are unavailable for this conditional implicit sampler, so this experiment evaluates randomization sensitivity and cost rather than bias correction.

\subsection{Experimental Protocol}
We consider three target distributions: (1) the 2D GMM $0.5\mathcal{N}(3\mathbf{e},I)+0.5\mathcal{N}(-3\mathbf{e},I)$, where $\mathbf{e}=(1,1)^\top$ and $\mathbb{E}[\|X\|^2]=20$; (2) MNIST, flattened to 784 dimensions, with pixel-mean functional $f(x)=784^{-1}\sum_{i=1}^{784}x_i$; and (3) a conditional distribution over five high-resolution vorticity frames given low-resolution dynamics and sparse noisy observations. We use RMSE over 20 randomizations for the GMM, randomization standard deviation for MNIST, and truth-based field metrics for vorticity assimilation. All reported RQMC results use independent Owen-style scrambling between randomizations.

\subsection{2D 2-Component GMM Experiment}
We use a ScoreNetMLP with 4 residual blocks, a hidden dimension of 256, and sinusoidal time embedding, trained for 1000 epochs with batch size 2048. The probability flow ODE is solved with 2000 Euler steps.
Quantitative results from Fig.~\ref{fig:convergence}a reveal two critical insights: First, importance sampling (IS) effectively mitigates transport bias: IS-augmented methods (Diffusion-ISRQMC, Diffusion-ISMC) exhibit sustained convergence, whereas non-IS counterparts (Diffusion-RQMC, Diffusion-MC) suffer from early saturation due to uncorrected bias. Second, RQMC sampling outperforms MC sampling universally: Across both IS and non-IS settings, RQMC delivers faster convergence--Diffusion-ISRQMC achieves a near-$\mathcal O(N^{-1+\varepsilon})$ rate ($\text{log-log slope}\approx -0.95$) with $\text{RMSE} = 0.002$ at $N=8192$, while Diffusion-ISMC follows the standard $\mathcal O(N^{-1/2})$ MC rate ($\text{slope}\approx -0.48$) with $\text{RMSE}\approx 0.08$ at the same sample size. These results confirm that the combination of IS and RQMC yields the most efficient convergence, with Diffusion-ISRQMC demonstrating the lowest RMSE and most stable performance.

\subsection{MNIST Dataset Experiment}
For MNIST (1-channel $28\times28$ images), we employ a ScoreNetUNet, a residual-based U-Net with sinusoidal time embedding and no attention modules, trained for 10,000 epochs with batch size 1024. The ODE is discretized with 2000 Euler steps.
We restrict our comparison to Diffusion-MC and Diffusion-RQMC. We exclude IS-based estimators because the curse of dimensionality leads to exponential weight variance growth and eventual weight collapse (zero ESS), rendering bias correction numerically unstable. Consequently, we focus on standard deviation to demonstrate RQMC's inherent variance reduction in the transport process.

From Fig.~\ref{fig:convergence}b, Diffusion-RQMC achieves a convergence rate near $\mathcal O(N^{-1+\varepsilon})$ and the lowest variance across all sample sizes. Diffusion-MC follows $\mathcal O(N^{-1/2})$ and maintains higher variance. RQMC thus delivers consistent variance reduction even when bias correction is unsuccessful.

To verify the validity of our sampling pipeline, we visualize generated MNIST samples in Fig.~\ref{fig:mnist_samples}. Both Diffusion-MC and Diffusion-RQMC produce reasonable handwritten digit samples that align with the MNIST data distribution. We do not emphasize visual quality differences between the two methods, as the advantage of RQMC over MC lies primarily in convergence rate and variance reduction of target function rather than sample visual fidelity.

\begin{figure}[!t]
    \centering
    \includegraphics[width=0.94\linewidth]{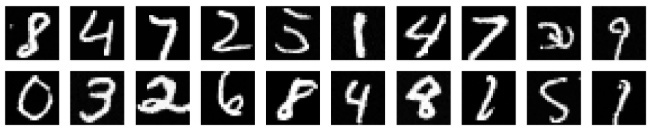}\\
    {\small (a) Diffusion-MC}\par
    \includegraphics[width=0.94\linewidth]{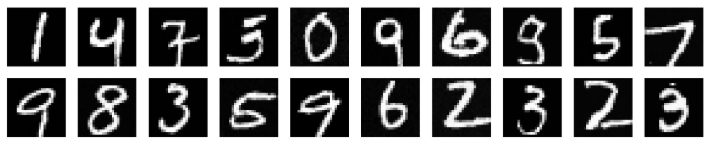}\\
    {\small (b) Diffusion-RQMC}
    \caption{Generated MNIST samples from Diffusion-MC and Diffusion-RQMC. Each panel displays two rows of handwritten digits.}
    \label{fig:mnist_samples}
\end{figure}

\subsection{Conditional Vorticity Data Assimilation}
\label{app:supp_weather}

\paragraph{Synthetic data and conditioning.}
Following the DiffSRDA setting \cite{ayapilla2026uncertainty}, we generate two-dimensional barotropic-jet trajectories with paired low- and high-resolution solvers. The full generation contains 5,000 trajectories and 445,000 aligned nine-frame windows. The low-resolution solver outputs $32\times17$ fields and the high-resolution solver outputs $128\times65$ fields; we discard the duplicated periodic endpoint, yielding learned high-resolution fields of size $128\times64$. Each example targets five high-resolution frames. Its condition contains three low-resolution frames separated by four solver-output intervals, bicubically upsampled to the target grid, and two sparse-observation frames from the first half of the window. Observations are placed every eighth grid point (128 sensors per frame) with additive Gaussian noise of standard deviation 0.1. We use 120,000 windows for training and 30,000 for validation.

\paragraph{Diffusion model and sampler.}
The conditional U-Net has 128,808,901 parameters, five state channels, base width 64, channel multipliers $(1,2,4,8,8)$, three residual blocks per level, and attention at resolution 8. It is trained with an $L_1$ diffusion objective and 1,000 training diffusion steps. We select epoch 266 by validation loss (training loss 0.01644; validation loss 0.03446). Evaluation uses deterministic DDIM with five reverse steps and $\eta=0$, so all sampling randomness is in the initial latent $z\in\mathbb{R}^{40{,}960}$.

PyTorch's Sobol implementation supports at most 21,201 dimensions. We therefore partition $z$ into contiguous blocks of dimensions 21,201 and 19,759, generate a separately scrambled Sobol net for each block, apply the Gaussian inverse CDF on the GPU, and concatenate the results. This independently scrambled \emph{blocked-Sobol'} construction is not a single 40,960D digital net; block interactions are not stratified. A fixed latent bank is generated once per randomization and reused over all contexts. IID MC uses the same pre-generation and caching protocol.

\paragraph{Evaluation and metrics.}
We evaluate 20 held-out conditioning contexts and 20 independent randomizations at ensemble size $N=256$.

Let $\bar\omega_{c,r,N,t}$ be an ensemble-mean frame and $\omega^\star_{c,t}$ its simulated truth. We report
\[
\operatorname{MAER}_{c,r,N,t}=
\frac{\operatorname{mean}_{x,y}|\bar\omega_{c,r,N,t}-\omega^\star_{c,t}|}
{\operatorname{mean}_{x,y}|\omega^\star_{c,t}|},
\]
the RMSE of an unscaled discrete Laplacian (periodic in the first spatial axis and replicated at the walls of the second), and MSSIM loss $1-\operatorname{SSIM}$. SSIM uses an $11\times11$ Gaussian window with standard deviation 1.5 and a per-truth-frame dynamic range. For a truth-based metric $M$, the stability curves report
\[
\operatorname{RSD}_t=
\left[\frac{1}{20}\sum_{c=1}^{20}
\operatorname{Var}_{r}\!\left(M_{c,r,256,t}\right)\right]^{1/2},
\]
with 95\% intervals obtained by bootstrapping contexts as clusters. Lower RSD means less sensitivity to the random seed or scramble.

\paragraph{Quantitative results.}

For truth-based nonlinear metrics, Table~\ref{tab:weather_stability} summarizes the reductions plotted in the main paper. MAER RSD falls by 47--66\% and MSSIM-loss RSD by 28--51\% at every lead. In contrast, the discrete-Laplacian RMSE has VRF 0.89--0.99 and therefore shows no benefit. This negative result helps localize the gain: the blocked construction stabilizes large-scale field and structural metrics, but not the most derivative-sensitive diagnostic.

\begin{table}[t]
\centering
\small
\setlength{\tabcolsep}{4.5pt}
\begin{tabular}{lrrrrr}
\toprule
Forecast lead $\tau$ & 0 & 0.25 & 0.50 & 0.75 & 1.00 \\
\midrule
MAER RSD reduction (\%) & 61.1 & 47.5 & 59.2 & 66.0 & 62.1 \\
MSSIM RSD reduction (\%) & 45.8 & 27.7 & 35.8 & 51.3 & 43.4 \\
\bottomrule
\end{tabular}
\caption{Reduction in randomization standard deviation from IID MC to blocked-Sobol' RQMC at $N=256$; larger is better.}
\label{tab:weather_stability}
\end{table}

\begin{table}[t]
\centering
\small
\setlength{\tabcolsep}{3pt}
\begin{tabular}{lrrr}
\toprule
Timing scope & IID (s) & RQMC (s) & Overhead \\
\midrule
Denoiser only & 2.2688 & 2.2687 & $0.00\%$ \\
Online, cached noise & 2.3058 & 2.3062 & $0.018\%$ \\
Full, amortized & 2.3088 & 2.3823 & $3.18\%$ \\
\bottomrule
\end{tabular}
\caption{Per-context wall time at $N=256$ on one NVIDIA H20. RQMC Gaussian noise is generated once per replicate and reused across contexts; IID uses the same timing boundary.}
\label{tab:weather_timing}
\end{table}

\paragraph{Qualitative fields.}
Figure~\ref{fig:weather_snapshots} shows three forecast leads for context 19, chosen by a deterministic representativeness ranking over the 20 contexts, rather than by visual inspection. Figure~\ref{fig:weather_uncertainty} compares ensemble means, spread, and the across-member standard deviation of absolute error at the first lead. The displayed uncertainty panels use one scramble; consequently, they illustrate that both methods reconstruct coherent jet structures and uncertainty concentration, not that every individual RQMC scramble must be closer to truth.

\begin{figure*}[t]
    \centering
    \includegraphics[width=0.98\textwidth]{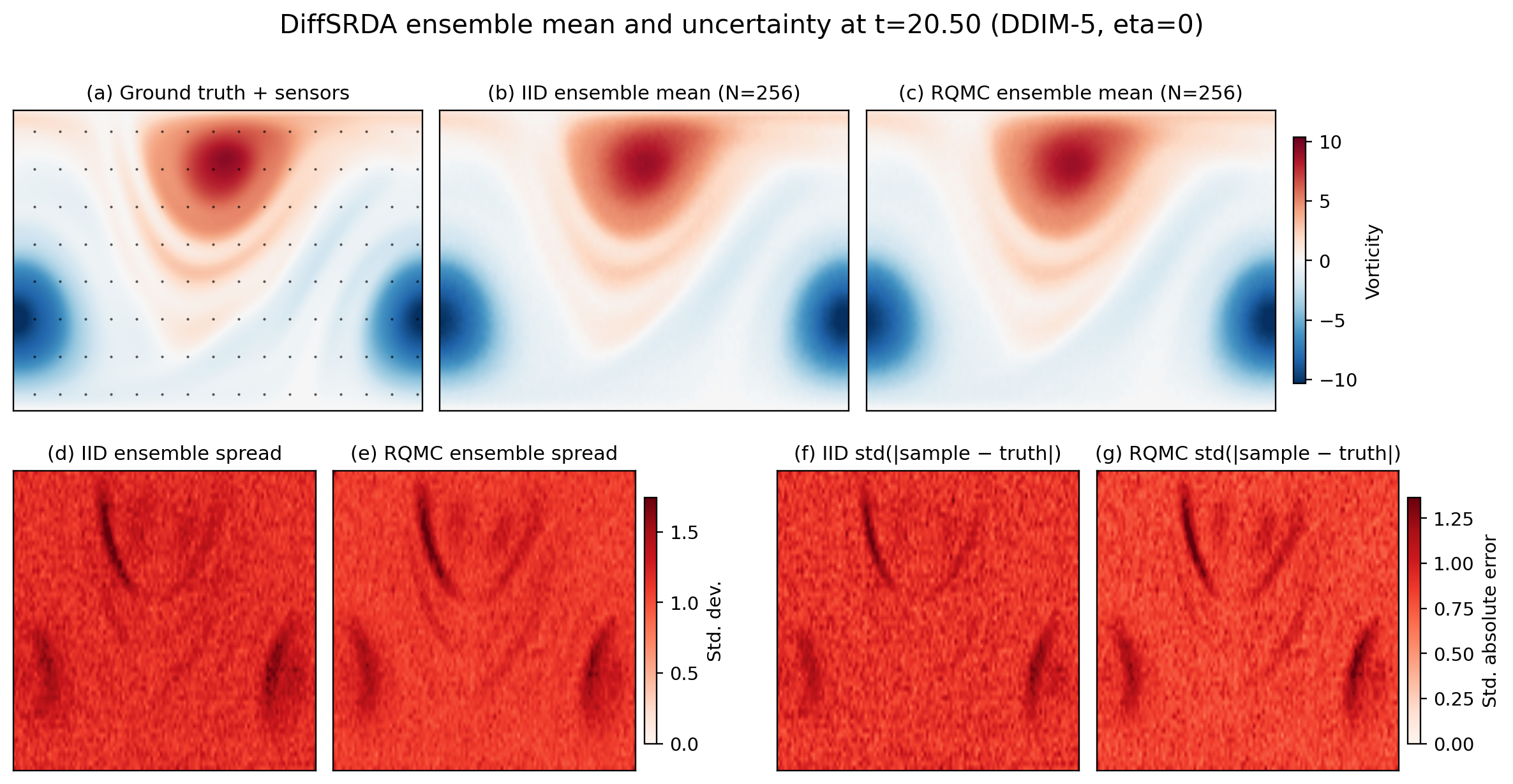}
    \caption{Ensemble mean and uncertainty at the first forecast lead. Top: simulated truth with sensors and the IID/RQMC ensemble means. Bottom: pixelwise ensemble spread and the across-member standard deviation of absolute error. These single-scramble fields are diagnostic visualizations; the aggregate randomization comparison is reported in the main paper and Table~\ref{tab:weather_stability}.}
    \label{fig:weather_uncertainty}
\end{figure*}

\subsection{Experimental Summary}
The experiments expose three complementary regimes. In the 2D GMM, IS removes transport bias and Diffusion-ISRQMC approaches the predicted $\mathcal{O}(N^{-1+\varepsilon})$ behavior, reaching roughly $40\times$ lower RMSE than ISMC at $N=2^{13}$. On MNIST, transport-only RQMC reduces the randomization variance of a pixel-mean functional despite unusable high-dimensional IS weights. In the 40,960D conditional vorticity task, independently scrambled Sobol blocks substantially stabilize MAER and MSSIM at near-identical online denoising cost, but do not improve the derivative-sensitive discrete-Laplacian metric. Thus the practical benefit depends on effective dimension and functional smoothness; it should not be inferred solely from ambient dimension or from the asymptotic theorem.

\section{Proof of the Derivative-Growth Result}
\label{app:proof_thm_diff_deriv_growth}

\begin{proof}
We prove the claim for the partial derivatives of the intermediate maps
\[
	\bdtau^{0:k} := \bdtau^{k-1}\circ\cdots\circ \bdtau^0 \circ G,
\qquad
	k=0,1,\dots,N,
\]
with $\bdtau^{0:0}=G$ and $\bdtau^{0:N}=\bdtau^{0:N}$.
We use induction on $k$.

\paragraph{Base case $k=0$ (only the inverse Gaussian CDF base map).}
Because $G$ is diagonal, $G_j(\bfu)=\Phi^{-1}(u_j)$.
Fix $m,j$ and $v\subseteq 1\!:d$.
Then $\partial^{\{m\}}G_j(\bfu)=0$ if $j\neq m$, and equals $(\Phi^{-1})'(u_m)$ if $j=m$.
Also $\partial^v(\partial^{\{m\}}G_j)$ is nonzero only if $v\subseteq\{m\}$.
By the inverse-Gaussian-CDF derivative bound from the main paper, with derivative orders 1 and 2, we have
\[
	\begin{aligned}
	\abs{(\Phi^{-1})'(u_m)}
	&\le
	C\,\min(u_m,1-u_m)^{-1}\\
	&\quad\cdot\abs{\ln(\min(u_m,1-u_m))}^{-1/2},\\
	\abs{(\Phi^{-1})''(u_m)}
	&\le
	C\,\min(u_m,1-u_m)^{-2}\\
	&\quad\cdot\abs{\ln(\min(u_m,1-u_m))}^{1/2}.
	\end{aligned}
\]
In either case, for any arbitrarily small $B>0$ we can upper bound the logarithmic factor by $\min(u_m,1-u_m)^{-B}$ (enlarge constants if needed). Hence
\[
	\begin{aligned}
		\abs{\partial^v(\partial^{\{m\}}G_j)(\bfu)}
		&\le
		C\, \min(u_m,1-u_m)^{-1}\\
		&\quad\cdot
		\prod_{k=1}^d \min(u_k,1-u_k)^{-B-\Indc{k\in v}},
	\end{aligned}
\]
which is exactly the required derivative-growth bound for $k=0$.

\paragraph{Inductive step.}
Assume $\bdtau^{0:k}$ satisfies the required derivative-growth bound. Consider
\[
	\bdtau^{0:k+1}(\bfu)
	= \bdtau^k\bigl(\bdtau^{0:k}(\bfu)\bigr)
	= \bdtau^{0:k}(\bfu) - h\,\bdv_\theta(\bdtau^{0:k}(\bfu),t_k).
\]
Fix $j\in 1\!:d$. Then
\[
	(\bdtau^{0:k+1})_j
	=
	(\bdtau^{0:k})_j - h\, v_{\theta,j}(\bdtau^{0:k}(\bfu),t_k).
\]
Taking $\partial^{\{m\}}$ and then $\partial^v$ gives
\begin{equation}\label{eq:inductive_split}
	\begin{aligned}
	&\partial^v\partial^{\{m\}}(\bdtau^{0:k+1})_j\\
	=&
	\partial^v\partial^{\{m\}}(\bdtau^{0:k})_j
	-
	h\,\partial^v\partial^{\{m\}}\Bigl(v_{\theta,j}\circ \bdtau^{0:k}\Bigr)(\bfu).
	\end{aligned}
\end{equation}
The first term is controlled by the inductive hypothesis. It remains to bound the second term.

\paragraph{Step 1: first derivative chain rule.}
Differentiate $v_{\theta,j}(\bdtau^{0:k}(\bfu),t_k)$ w.r.t.\ $u_m$:
\[
	\begin{aligned}
	&\partial^{\{m\}}(v_{\theta,j}\circ \bdtau^{0:k})\\
	=&
	\sum_{p=1}^d
	\Bigl(\partial_{x_p} v_{\theta,j}\Bigr)(\bdtau^{0:k}(\bfu),t_k)\cdot
	\partial^{\{m\}}(\bdtau^{0:k})_p(\bfu).
	\end{aligned}
\]

\paragraph{Step 2: apply $\partial^v$ and product rule.}
Applying $\partial^v$ yields a sum over $p$ and partitions $v=v_1 \sqcup  v_2$:
\[
	\begin{aligned}
		\partial^v\partial^{\{m\}}(v_{\theta,j}\circ \bdtau^{0:k})
		&=
		\sum_{p=1}^d \sum_{v_1 \sqcup  v_2=v}
		\partial^{v_1}\Bigl[(\partial_{x_p} v_{\theta,j})\circ \bdtau^{0:k}\Bigr]\\
		&\quad\cdot
		\partial^{v_2}\Bigl[\partial^{\{m\}}(\bdtau^{0:k})_p\Bigr].
	\end{aligned}
\]

\paragraph{Step 3: bound each factor.}
By the bounded-vector-field assumption in the main paper, all spatial derivatives of $v_{\theta,j}$ are bounded. The composition derivative
$\partial^{v_1}[(\partial_{x_p} v_{\theta,j})\circ \bdtau^{0:k}]$
can be expanded via Fa\`a di Bruno; every term is a product of bounded derivatives of $\partial_{x_p}v_{\theta,j}$ with derivatives of $\bdtau^{0:k}$.
Since bounded factors do not change boundary singularity order, the growth is dominated by products of derivatives of $\bdtau^{0:k}$.
Using the inductive hypothesis for $\bdtau^{0:k}$ repeatedly, we obtain for arbitrarily small $B>0$ a bound of the form
\[
	\begin{aligned}
	&\abs{\partial^{v_1}\Bigl[
	(\partial_{x_p} v_{\theta,j})\circ \bdtau^{0:k}\Bigr]}\\
	&\quad\le
	C\prod_{r=1}^d
	\min(u_r,1-u_r)^{-B-\Indc{r\in v_1}}.
	\end{aligned}
\]
More explicitly, every Fa\`a di Bruno term is a finite product of two types of factors: (i) spatial derivatives of $\partial_{x_p}v_{\theta,j}$ evaluated at $\bdtau^{0:k}(\bfu)$, which are uniformly bounded by that assumption, and (ii) derivatives of the inner map $\bdtau^{0:k}(\bfu)$, which satisfy the boundary growth bound by the inductive hypothesis. Therefore each term inherits its boundary singularity entirely from the inner-map derivatives, and summing finitely many such terms yields the stated bound.

For the second factor, the inductive hypothesis yields
\[
	\begin{aligned}
		\abs{\partial^{v_2}[\partial^{\{m\}}(\bdtau^{0:k})_p]}
		&\le
		C\,\min(u_m,1-u_m)^{-1}\\
		&\quad\cdot
		\prod_{r=1}^d \min(u_r,1-u_r)^{-B-\Indc{r\in v_2}}.
	\end{aligned}
\]

\paragraph{Step 4: multiply and recombine exponents.}
Multiplying the two factors and using $v_1\sqcup v_2=v$, we obtain
\[
	\begin{aligned}
		\abs{\partial^v\partial^{\{m\}}(v_{\theta,j}\circ \bdtau^{0:k})}
		&\le
		C\,\min(u_m,1-u_m)^{-1}\\
		&\quad\cdot
		\prod_{r=1}^d \min(u_r,1-u_r)^{-B-\Indc{r\in v}}.
	\end{aligned}
\]
Substituting back into \eqref{eq:inductive_split} (absorbing the factor $h$ into constants), we conclude that $\bdtau^{0:k+1}$ satisfies the same growth bound, completing the induction.
\end{proof}